\documentclass[a4paper ]{article}
\usepackage{amsthm}
\usepackage{amsmath,amssymb}
\usepackage{graphicx}

\newtheorem{theorem}{Theorem}[section]
\theoremstyle{plain}
\newtheorem{conjecture}[theorem]{Conjecture}
\newtheorem{lemma}[theorem]{Lemma}

\newtheorem{corollary}[theorem]{Corollary}
\newtheorem{observation}[theorem]{Observation}
\newtheorem{problem}[theorem]{Problem}
 
\theoremstyle{definition}
\newtheorem{definition}[theorem]{Definition}
\newtheorem{example}[theorem]{Example}

\def\P{\mathcal{P}_q}

\begin{document}

\title{From the Ising and Potts models to the general graph homomorphism polynomial}

\author{Klas Markstr\"om }
\maketitle

%
%

\section{Introduction}
A graph homomorphism from a graph $G$ to a graph $H$ is a mapping $h:V(G)\rightarrow V(H)$ such that $h(u)\sim h(v)$ if $u\sim v$. Graph homomorphisms are 
well studied objects and, for suitable choices of either $G$ or $H$, many classical graph properties can be formulated in terms of homomorphisms. For example the 
question of wether there exists a homomorphism from $G$ to $H=K_q$ is the same as asking wether $G$ is $q$-colourable or not.  A number of classical models 
in statistical physics, such as the Ising model, Potts model and lattice gas, can be formulated in terms of the generating function for weighted versions of homomorphisms 
from $G$ to some graph $H$.  We refer the reader to \cite{HN} for a comprehensive survey of the algebraic aspects of graph homomorphisms.

Our aim here is to discuss the generating polynomial for homomorphisms from a graph $G$ to the most general weighted graph on $q$ vertices.  For a fixed $q$ 
this is an object of polynomial size  which contains a wealth of informations about the graph $G$, but as we will later show it is not a complete graph invariant.  We will first 
defined this generating function as a polynomial, then recall the definitions of a number of well known graph polynomials, and partition functions from physics, and then proceed to study the properties and relationships of these polynomials.

Let us give the formal definitions of our objects of study.
\begin{definition}
	Given a weighted graph $H$ with a weight function $w$ which assigns a weight to each edge and vertex of $H$ and a homomorphism $\phi$ from a graph
	 $G$ to $H$ we define the weight of $\phi$ to be
	$$w(\phi) =\prod_{v\in V(G)} w(\phi(v))  \prod_{e\in E(G)} w(e) $$	
\end{definition}
We let  $\textrm{Hom}(G,H)$ denote the set of all homomorphisms from $G$ to $H$.

\begin{definition}\label{Id2}
    Let $W_q$ be a weighted complete graph on $q>1$ vertices where vertex 
    $i$ has weight $x_{i}$ and an edge ${i,j}$ gets weight $y_{ij}$, where each weight is a formal variable.
    Note the $i$ and $j$ may be equal.
   
    Let $$\P(G)=\sum_{\phi\in \textrm{Hom}(G,W_q)}w(\phi)$$
    We call $\mathcal{P}_q(G)$ the \emph{homomorphism polynomial} of order $q$ of $G$, or simply the homomorphism polynomial of $G$ when the order is clear from the context.
\end{definition}

\begin{lemma}
      If $G$ has $n$ vertices and $m$ edges then  $\mathcal{P}_q(G)$  is a polynomial in $2q+{q \choose 2}$ variables and each monomial of the polynomial has 
      total degree at most $n$ in the $x_i$'s and at most $m$ in the $y_{i,j}$'s.
      
       $\mathcal{P}_q(G)$  is a symmetric polynomial in the variables  $x_i,\ldots,x_q$.
\end{lemma}

\begin{example}
	The homomorphism polynomial of order $q=3$ for the three vertex path $P_3$ is
	\begin{multline}
		\mathcal{P}_3(P_3)=x_1^3 y_{1,1}^2 + 2 x_1^2 x_2 y_{1,1} y_{1,2} + x_1^2 x_2 y_{1,2}^2 + 	x_1 x_2^2 y_{1,2}^2 +  2 x_1^2 x_3 y_{1,1} y_{1,3} + \\ 
		2 x_1 x_2 x_3 y_{1,2} y_{1,3} + x_1^2 x_3 y_{1,3}^2 +  x_1 x_3^2 y_{1,3}^2 + 2 x_1 x_2^2 y_{1,2} y_{2,2} + \\x_2^3 y_{2,2}^2 + 
		2 x_1 x_2 x_3 y_{1,2} y_{2,3} + 2 x_1 x_2 x_3 y_{1,3} y_{2,3} + 2 x_2^2 x_3 y_{2,2} y_{2,3} +  \\ x_2^2 x_3 y_{2,3}^2 + 
		x_2 x_3^2 y_{2,3}^2 + 2 x_1 x_3^2 y_{1,3} y_{3,3} +  2 x_2 x_3^2 y_{2,3} y_{3,3} + x_3^3 y_{3,3}^2
	\end{multline}
\end{example}

\begin{lemma}\label{com}
	The homomorphism polynomial of order $q$ for the complete graph on $n$ vertices can be computed in time $\mathcal{O}(n^q)$ and is given by

	$$\P(K_n)=\sum_{(n_1,\ldots,n_q)\in S(n,q)}\prod_{i=1}^{q} x_i^{n_i}\prod_{i=1}^{q}x_{ii}^{{n_i \choose 2}}\prod_{i<j}x_{ij}^{n_in_j}
	$$
	Here $S(n,q)$ is the set of ordered partitions into $q$ nonnegative integers  of $n$ .
\end{lemma} 

\section{Graph polynomials and partition functions}
In this section we will recall the definitions for several graph polynomials and the partition functions for different models in statistical physics which in different 
ways relate to the graph homomorphism polynomial.

The first of these  is the \emph{chromatic polynomial} of a graph, first introduced by Whitney \cite{Whi} as part of his work on the four-colour conjecture.
\begin{definition}
	$C(G,q)=\textrm{The number of proper } $q-$\textrm{colourings of } G$.
\end{definition}
It is no immediate that this at all a polynomial, but as shown by Whitney for a graph $G$ with $n$ vertices this is in fact a monic polynomial of degree $n$.

The chromatic polynomial is a specialization of the well kown Tutte-polynomial, first introduced by Tutte in \cite{Tu:47} as the most general graph polynomial 
satisfying certain relations when edges are contracted or deleted.  This polynomial has been studied in many different parametrisations,  related by  variable substitutions,
 and the one which we will state is the one connected to the random-cluster model.
\begin{definition}
	Given a graph $G$ with $n$ vertices and $m$ edges:	
	$$T(G, p,q)=\sum_{A\subset E(G)} p^{|A|}(1-p)^{m-|A|}q^{k(A)},$$
	where $k(A)$ denotes the number of connected components in the subgraph induced by vertices in $G$ and the edges in $A$
\end{definition}
By setting $q$ to a fixed integer in $T(G,p,q)$  we obtain the \emph{$q$-state Potts partition function}  $Z_q(G,p)$, which is the partition function for the 
Potts model, one of the most studied models in statistical physics.  In the special case where $q=2$ this also called the Ising polynomial, and is the partition 
function for the Ising model.  We refer the reader to \cite{Grim} for a book length treatment of the probabilistic aspects of these models, and to \cite{sok} for the 
more combinatorial aspects.  

There is also a two-variable polynomial connected to the ising polynomial. We will here follow the treatment in \cite{AM}.

A \emph{state} $\sigma$ on $G$ is a function $\sigma 
:V(G)\rightarrow\{ -1,1 \}$, the value of $\sigma$ at a vertex $v$ is 
called the \emph{magnetisation} of $v$. 
\begin{definition}
    Given a state $\sigma$ the energy $E(\sigma,e)$ of an edge 
    $e=(u,v)$ in $G$ is $E(\sigma,e)=\sigma(u)\sigma(v)$, and the 
    energy $E(\sigma) $of the state $\sigma$ is the sum of the energies 
    of the edges, that is 
    $$E(\sigma)=\sum_{e\in E(G)}E(\sigma,e).$$
\end{definition}
\begin{definition}
    The magnetisation $M(\sigma)$ of a state $\sigma$ is the sum of 
    the magnetisations of all the vertices in $G$, that is,
    $$M(\sigma)=\sum_{u \in V(G)}\sigma(u).$$
\end{definition}Let $\Omega$ denote the set of all states on $G$.

We can now define the \emph{Bivariate Ising polynomial}:
\begin{definition}\label{Id1}
    The Bivariate Ising polynomial is  
    $$Z(G,x,y)=\sum_{\sigma\in \Omega}x^{E(\sigma)}y^{M(\sigma)}=
    \sum_{i,j}{a_{i,j}x^i y^j}.$$
\end{definition}
The Ising polynomial mentioned earlier can be obtained by setting $y=1$ in the bivariate Ising polynomial, substituting $x$ for a rational function in $p$ and expanding. We refer to \cite{AM} for full discussion of this.

The bivariate Ising polynomial is also closely connected to the \emph{van der Waerden-polynomial}, a polynomial studied van der Waerden in \cite{VDW}.
\begin{definition}[van der Waerden Polynomial]
    $$W(G,t,u)=\sum_{i,j}b_{i,j}u^i t^j,$$
    where $b_{i,j}$ is the number of subgraphs of $G$ with $i$ edges 
    and $j$ vertices of odd degree.
\end{definition}
As shown in \cite{AM} the bivariate Ising polynomial and the van der Waerden-polynomial are equivalent and are related to each other by a change of variables.

A specialisation of the van der Waerden-polynomial is the \emph{matching polynomial}.
\begin{definition}
    The matching polynomial of a graph $G$ is  
    $$m(G,x)=\sum_{i} m_i x^i,$$
    where $m_i$ is the number of matching with $i$ edges in $G$.
\end{definition}
This polynomial corresponds to the part of the van der Waerden-polynomial which counts subgraphs with $i$ edges and $2i$ vertices of odd degree.

We finally have two more polynomials which are partition functions for related models from statistical physics.  First we have the \emph{Independence polynomial}:
\begin{definition}
    The independence polynomial of a graph $G$ is  
    $$I(G,x)=\sum_{i} c_i x^i,$$
    where $c_i$ is the number of independent sets with $i$ vertices in $G$.
\end{definition}
If we take the independence polynomial of the line graph $L(G)$ of $G$ we obtain the matching-polynomial of $G$, i.e. $I(L(G),x)=m(G,x)$.

The independence polynomial is a specialisation of the partition function for the \emph{hard-core lattice gas}.
\begin{definition}
    $$H(G,h,x)=\sum_{A\subset V(G)} h^{||G[A]||}x^{|A|},  $$
    where $||G[A]||$ denotes the number of edges in the subgraph of $G$ induced by the vertex set $A$. 
\end{definition}
The independence polynomial corresponds to the part of this partition function where the exponent of $h$ is 0, or likewise $I(G,x)=H(G,0,x)$.  In \cite{Bax} the reader can 
find a treatment of the analytical and statistical physics side of the lattice gas model for different graphs.

\section{Basic properties of the homomorphism polynomial}
The homomorphism polynomial satisfies a number of relationships similar to those which hold for other classical graph polynomials.

\begin{theorem}\label{th1}
    If $G$ has components $G_{1}$ and $G_{2}$ then    $\P(G)=\P(G_{1})\P(G_{2})$.
\end{theorem}
\begin{proof}
    Immediate by Definition \ref{Id2} since any homomorphism $\phi$ from $G$ to $W_q$ consists of two independent homomorphisms, $\phi_1\in \textrm{Hom}(G_1,W_q)$,   
     $\phi_2\in \textrm{Hom}(G_2,W_q)$, and  $w(\phi)=w(\phi_1)w(\phi_2)$.
\end{proof}
The Tutte-polynomial can be factored in the same way for disconnected graphs, and  it is also possible to factor the Tutte-polynomial as $T(G,x,y)=T(G_1,x,y)T(G_2,x,y)$ if $V(G)=V(G_1)\cup V(G_2)$ and $V(G_1)\cap V(G_2)$ is a single vertex.   As a consequence it follows that all trees on $n$ vertices have the same Tutte-polynomial. However, as we will later see, there are trees with the same number of vertices but distinct homomorphism polynomials an hence no similar vertex cut expression exists for $\P(G)$.

\begin{theorem}\label{th2}
    The homomorphism  polynomial of $G$ determines the homomorphism  polynomial of the complement $\overline{G}$.
\end{theorem}
\begin{proof}
	Assume that $G$ has $n$ vertices. Given the exponents of the variables $x_1,\ldots,x_q$ in a monomial of $\P(G)$ we can compute what the exponents of the 
	$y$-variables would be in the complete graph on $n$ vertices, and the sum of the exponents of $y_{ij}$ in $\P(G)$ and $\P(\overline{G})$ is equal to the exponent in $\P(K_n)$.
\end{proof}
The Tutte-polynomial does not have this property, as can be seen by considering trees and their complements, but e.g the matching polynomial does \cite{God}.

The \emph{join} of two graphs $G_{1}$ and $G_{2}$ is the graph obtained by taking the disjoint union of the two graphs and
adding an edge from every vertex in $G_{1}$ to every vertex in $G_{2}$.
\begin{corollary}
    If $G$ is the join of two graphs $G_{1}$ and $G_{2}$, then $\P(G)$ can be constructed in polynomial time from $\P(G_1)$ and $\P(G_2)$.
\end{corollary}
\begin{proof}
	This follows directly from Theorems \ref{th1} and \ref{th2} since the join of $G_1$ and $G_2$ is isomorphic to $\overline{ \overline{G_1} \cup \overline{G_2}}$.
\end{proof}

As an application of the previous corollary we can extend Lemma \ref{com} to compete multipartite graph as well.
\begin{corollary}
	The homomorphism polynomial of the complete multipartite graph $K_{t_1,t_2,\ldots,t_k}$ can be constructed in polynomial time. 
\end{corollary}
\begin{proof}
	This follows from Theorems \ref{th1} and \ref{th2} since the complete multipartite graph is isomorphic to $\overline{K_{t_1}\cup\ldots\cup K_{t_k}}$.
\end{proof}

Finally we note that if we increase the value of $q$ all information from lower values remain.
\begin{lemma}\label{mon}
	$\mathcal{P}_q(G)$ determines $\mathcal{P}_{q'}(G)$ for all $q'<q$.
\end{lemma}
\begin{proof}
	$\mathcal{P}_{q-1}(G)$ is the polynomial given by those monomials in $\mathcal{P}_q(G)$ with 0 for $x_q$, and the result follows inductively for smaller $q'$.
\end{proof}

\section{Graph properties, counting problems, and other graph polynomials}
Many of the  graph polynomials already studied in both  mathematics and physics are special cases of the homomorphism polynomial, for specific values of $q$ and the weight-variables.
\begin{theorem}\label{th3}
	The following polynomials are determined by $\mathcal{P}_q(G)$ for all $q\geq 2$.
	\begin{enumerate}
		\item The Ising polynomial (The partition function for the Ising model without external field).
		\item The Bivariate Ising polynomial  $Z(G,x,y)$ (The partition function for the Ising model with external field).
		\item The Matching polynomial (The partition function for dimer coverings).
		\item The van der Waerden polynomial.
		\item The Independence  polynomial.
		\item  The hard-core lattice gas partition function. 
	\end{enumerate}
\end{theorem}
\begin{proof}
	From Lemma \ref{mon} we know that we can determine $\mathcal{P}_2(G)$ from $\mathcal{P}_q(G)$ and $\mathcal{P}_2(G)$ directly gives us the bivariate Ising 
	polynomial $Z(G,x,y)$. In \cite{AM} it was proven that $Z(G,x,y)$ also determines 1,3 and 4.
	
	The independent set polynomial is determined by the hard-core partition function and it is in turn obtained from $\mathcal{P}_2(G)$ by setting $x_2=y_{22}=y_{12}=1$.
\end{proof}
Here we can mention a problem for further investigation. In \cite{AM} it is proven that the bivariate Ising polynomial and the van der Waerden polynomials are equivalent, 
via a pair of variable substitutions.
\begin{problem}
	Develop an analogue of the variable substitutions for the bivariate Ising and van der Waerdens polynomials for  $\P(G)$ for general $q$.
\end{problem}

From the fact that $\mathcal{P}_q(G)$ determines the bivariate Ising polynomial it follows that many of the basic properties of a graph are determined by $\mathcal{P}_q(G)$.
In \cite{AM} it was shown that the bivariate Ising polynomial, and hence $\P(G)$, determines the following properties of a graph.
\begin{theorem}[\cite{AM}]
Z(G,x,y) determines the following properties of a graph $G$.
  \begin{enumerate}
	\item The order $n$, size $m$, and the degree-sequence of $G$.
	\item The number of components of $G$ and their size.
	\item The smallest edge-connectivity of the components of $G$.
	\item The size of a maximal edge-cut in $G$, i.e. the largest value of $\left|[S,\bar{S}]\right|$.
	\item Whether $G$ is bipartite or not.
	\item The girth of $G$.	
	
	\item For an $r$-regular graph $G$, the number of $k$-cliques and the number of independent sets of  size $k$.
	\item Whether $G$ is a tree or not, and if $G$ is a tree its diameter and its characteristic polynomial.
	
    \end{enumerate}
\end{theorem}

The restriction to regular graphs in 7 is not necessary for $\P(G)$
\begin{theorem}
	$\P(G)$ determines the number of $k$-cliques and the number of independent sets of size $k$ in $G$.
\end{theorem}
\begin{proof}
	By Theorem \ref{th3} the generating function for the number of independent sets of all sizes is determined by $\P(G)$, and the corresponding generating 
	function for cliques can be obtained from $\P(\overline{G})$ by Theorem \ref{th2}.
\end{proof}

In a series of papers \cite{SS1,SS2,SS3,SS4,SS5} Scott and Sorkin studied a very general class of constraint satisfaction problems, CSPs, for graphs. The focus of those papers   is algorithmic, and primarily on binary CSPs, finding fast algorithms for various versions of the CSPs on both sparse random graphs and graphs of bounded tree-width. In the series, they introduced a polynomial version of the CSPs where formal variables were assigned as weights to the edges and vertices of the graph and to the constraints as well.  While their emphasis was on situations with general weights on the graph and not on the constraint we note that the homomorphism polynomial can be obtained as a special case of their  polynomially weighted CSP by having weight 1 on all edges of the graph and general weights  on the constraints. For a more detailed description see \cite{SS5}.

\section{Non-isomorphic graphs with the same homomorphism polynomial}
One question which has been studied in connection with several of the graph polynomials already mentioned is whether non-isomorphic graphs can have the same 
polynomial, so far the answer has been yes for all polynomials studied, and how finely the given polynomial partitions the set of all graphs into equivalence classes.  For the bivariate Ising polynomial this was done in \cite{AM}, where arbitrarily large sets of non-isomorphic graphs with the same bivariate Ising polynomial were constructed, and in \cite{GGN2} this was done for the Potts model partition function. This is a well studied question for the chromatic polynomial, see e.g. \cite{Tu}.

For $2\leq q \leq 4$ we have carried out a complete classification of the graphs on $n\leq 10$ vertices according to their $\P(G)$, by a direct computer search.   We give more details about the computer search in Appendix \ref{comp} and present the results here. The smallest non-isomorphic graphs with the same $\mathcal{P}_2(G)$  have 8 vertices, see Figure \ref{fig:iso2} for an example. In Table \ref{tab:iso2} we show the number of non-trivial equivalence classes for each $n$.

\begin{figure}
    \includegraphics[width=0.5\textwidth]{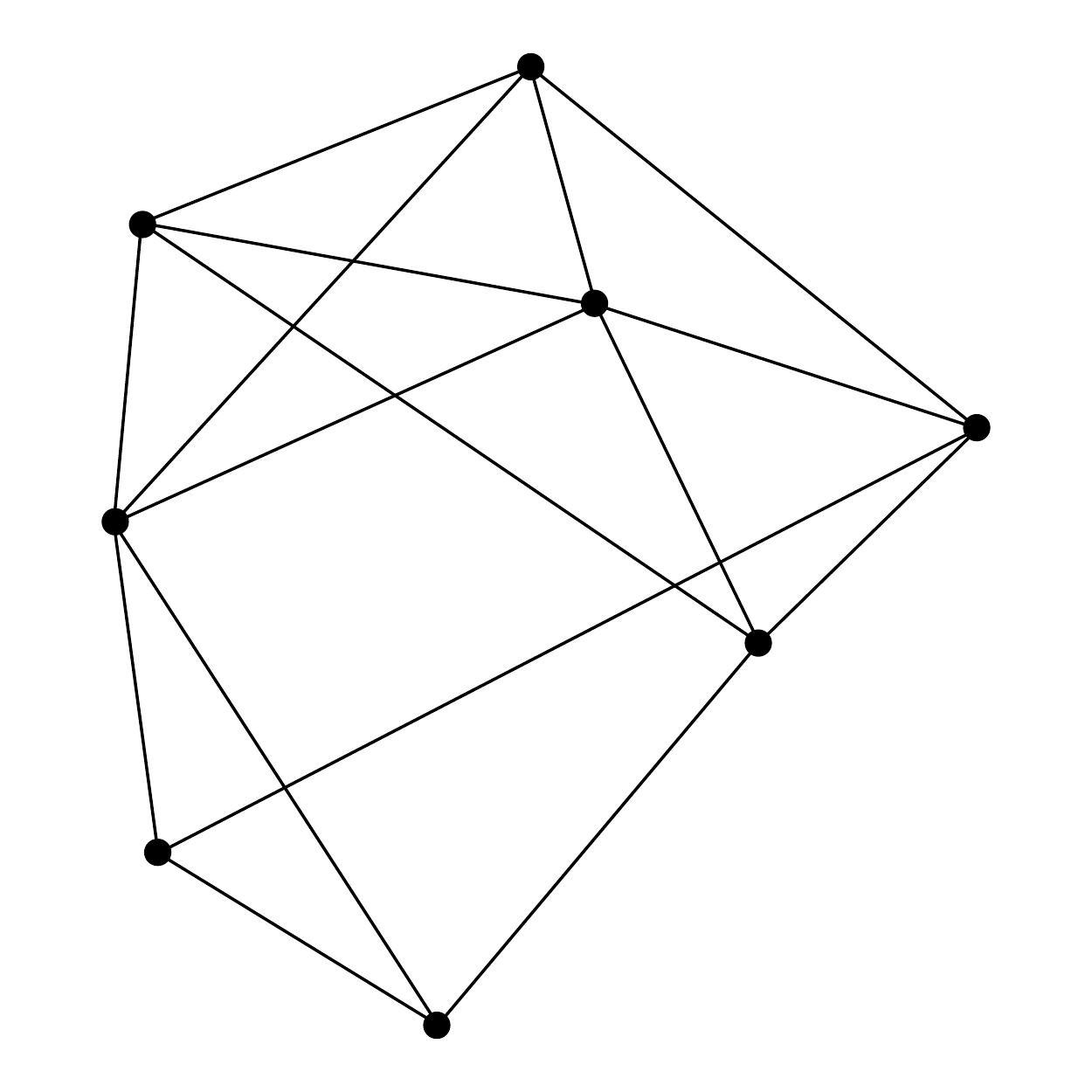}
    \includegraphics[width=0.5\textwidth]{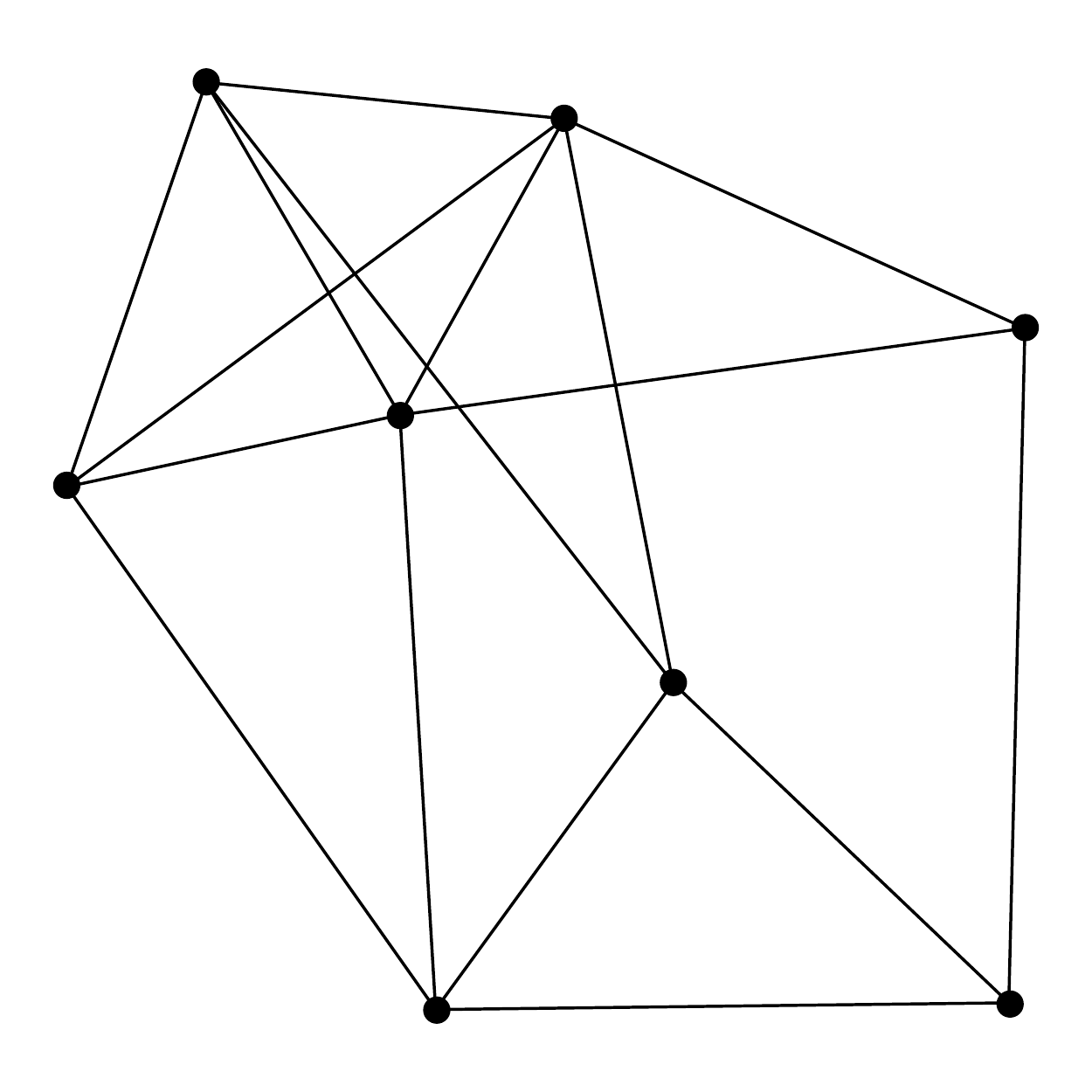}
    \caption{Two 8-vertex graphs with $\mathcal{P}_2(G_1)=\mathcal{P}_2(G_2)$.}\label{fig:iso2}
\end{figure}

\begin{table}[h]
\begin{center}
    \begin{tabular}{|r|r|r|r|}
	\hline
	$n$ & $q=2$ & $q=3$ & $q=4$  \\
	\hline
	1-7 &  0       &  0  & 0 \\
	8 &    29       &  0 & 0 \\
	9 &   2200   &  0 & 0 \\
       10 &  29270   & 2  & 0 \\

	\hline
    \end{tabular}
    \caption{The number of non-trivial equivalence classes for $q=2,3,4$. }\label{tab:iso2}
    \end{center}
\end{table}

The smallest graphs with the same $\mathcal{P}_3(G)$ have 10 vertices, and the non-trivial equivalence classes are two pairs of graphs, one of which is displayed in Figure \ref{fig:iso3a} and the other is the complements of the graphs in the displayed pair.
\begin{figure}
    \includegraphics[width=0.5\textwidth]{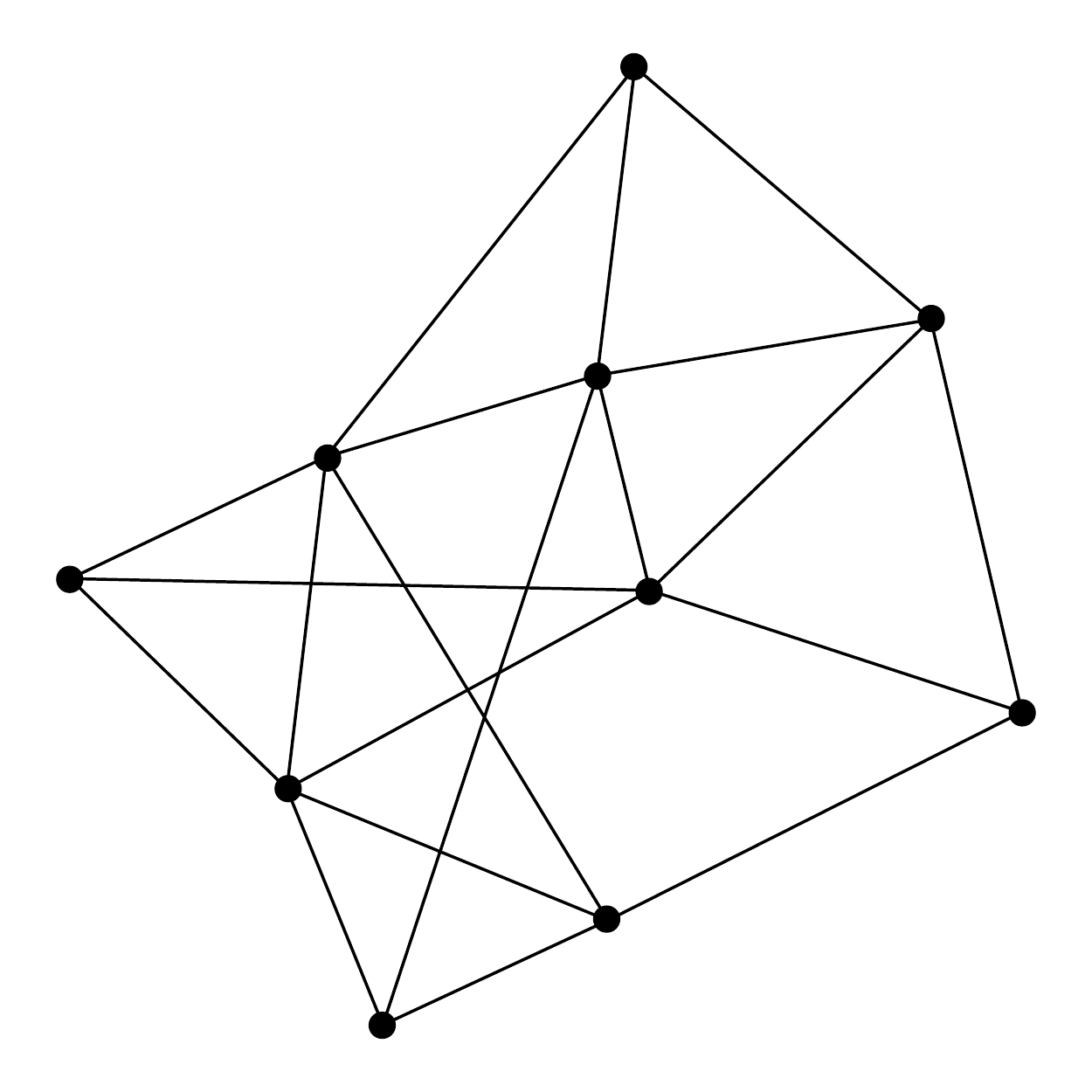}
    \includegraphics[width=0.5\textwidth]{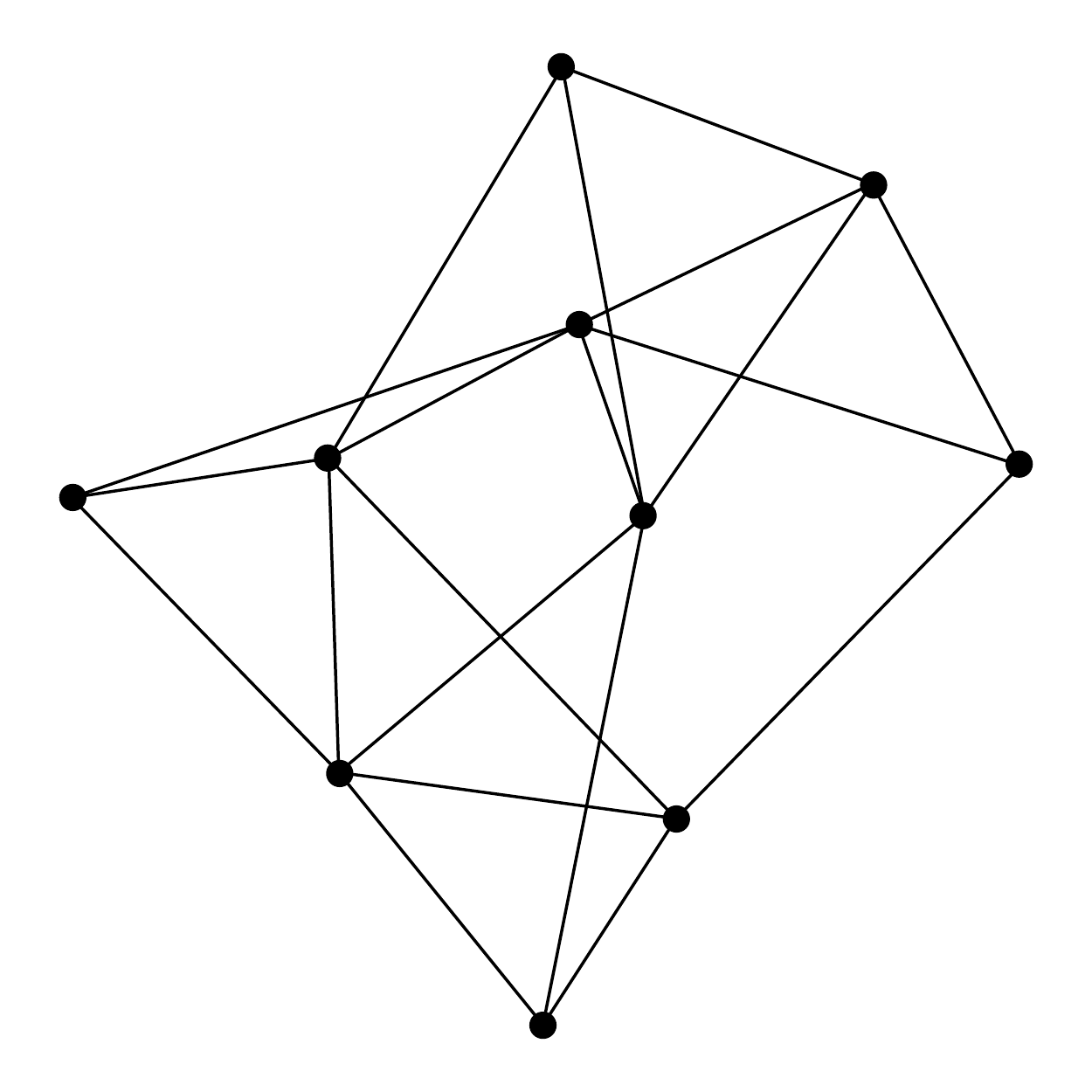}

    \caption{Two 10-vertex graphs with $\mathcal{P}_3(G_1)=\mathcal{P}_3(G_2)$.}\label{fig:iso3a}
\end{figure}

As Lemma \ref{mon} shows the property of having $\P(G)$ distinct from all other graphs is monotone in $q$, i.e., if $\P(G_1)=\P(G_2)$ then the homomorphism polynomial 
for lower values of $q$ will also be equal. If some of the variables in $\P(G)$ are given either numerical values or given equal weights this type of monotonicity  is not necessarily preserved. In \cite{GGN2} the authors asked for examples of graphs with the same $q$-state Potts model partition function for $q=3$ but different ones for $	q=2$. We used the same computer programs as for our earlier classification to look for such examples as well. This can be done easily since the Potts model partition function is a specialization of $\P(G)$.  The smallest examples of this type have 8 vertices. In Figure \ref{fig:potts} we display two 8-vertex graphs with the same $3$-state Potts model partition function and distinct $2$-state Potts model partition functions.

\begin{figure}
    \includegraphics[width=0.5\textwidth]{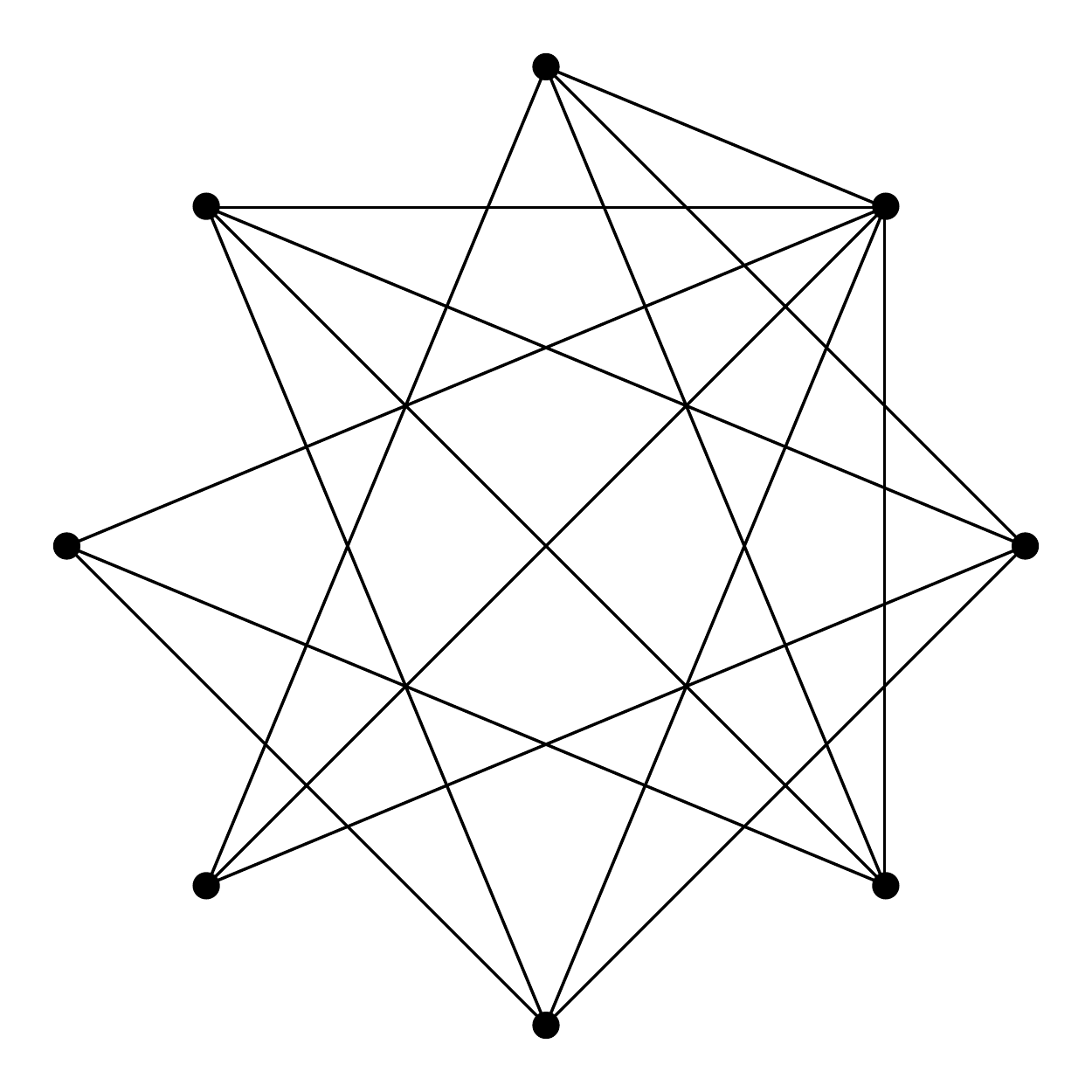}
    \includegraphics[width=0.5\textwidth]{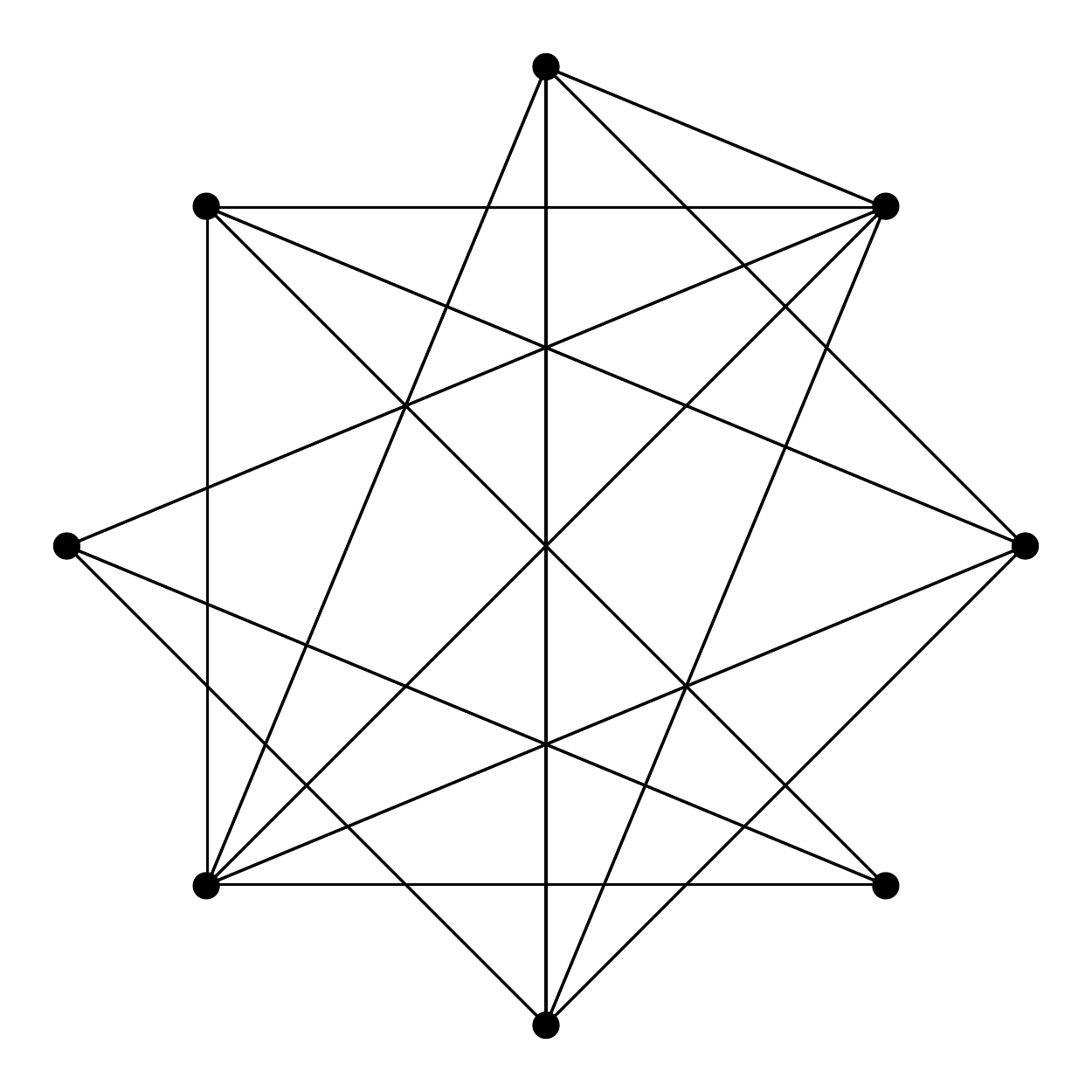}
    \caption{Two 8-vertex graphs with the same $q$-state Potts model partition function for $q=3$ but distinct ones for $q=2$.}\label{fig:potts}
\end{figure}

\subsection{Families of graphs with the same homomorphism polynomial}
Given the relatively large number of small graphs with the same homomorphism polynomial for $q=2$ and small such number for $q=3$ it is natural to ask if there is some $q$ which does determine all graphs and otherwise if the size of the smallest examples are strictly increasing in $q$. 

\begin{lemma}
	If $G_1$ and $G_2$ are graphs on $q$ vertices and $\P(G_1)=\P(G_2)$ then $G_1$ and $G_2$ are isomorphic.
\end{lemma}
\begin{proof}
	If $\P(G_1)=\P(G_2)$ then there exists a monomial which corresponds to an isomorphism from $G_1$ to a labelled copy of $G_1$  in the weight graph $W_q$. From the identity it follows that there also exists a homomorphism from $G_2$ to the same labelled copy of $G_1$, and since it is a surjection of the vertices it is an isomorphism.    
\end{proof}

From this it follows that the rotor-type construction used in \cite{Tu} for the Tutte-polynomial and in \cite{AM} for the bivariate Ising polynomial cannot be used in order to show that for every $q$ there are graphs with the same homomorphism polynomial.  However we strongly believe that this is the case.
\begin{conjecture}
	For any fixed $q$ there exist non-isomorphic graphs $G_1$ and $G_2$ such that $\P(G_1)=\P(G_2)$.
\end{conjecture}

\begin{problem}
	Let $f(q)$ be the smallest integer such that there exist non-isomorphic graphs $G_1$ and $G_2$ on $f(q)$ vertices such that $\P(G_1)=\P(G_2)$.  Determine the rate of growth of $f(q)$.
\end{problem}
The preceding lemma demonstrates that $f(q)>q$ and our earlier examples show that $f(2)=8$, $f(3)=10$, and $f(4)>f(3)$. It follows from Lemma \ref{mon} that $f(q)$ is monotone but we do not know if it is strictly monotone.
\begin{conjecture}
	$f(q)<f(q+1)$. 
\end{conjecture}

\section{Colourings, the Tutte-polynomial and  its generalizations.}

For $k\leq q$ we can determine the number of proper vertex $k$-colourings of $G$ from $\P(G)$, and we can count the number of proper colourings with colour classes of given sizes as well.  However,  if $q<n-1$ then $\P(G)$ does not always determine the chromatic polynomial of $G$, and hence not the Tutte-polynomial either.  In Figure \ref{fig:chro} we display two graphs with the same homomorphism polynomial for $q=2$ but distinct chromatic polynomials. For the $q$-state Potts partitions function this can be strengthened even further, in Figure \ref{fig:chro2} we display two graphs with the same  $q$-state Potts partitions function for $q=2,3$ but with distinct chromatic polynomials.
   If $q\geq n-1$ then the chromatic polynomial is determined by  $\P(G)$, since it determines the number of $k$-colourings of $G$ for $1 \leq k \leq n-1$, and the chromatic polynomial is a monic polynomial of degree $n$.
\begin{figure}
    \includegraphics[width=0.5\textwidth]{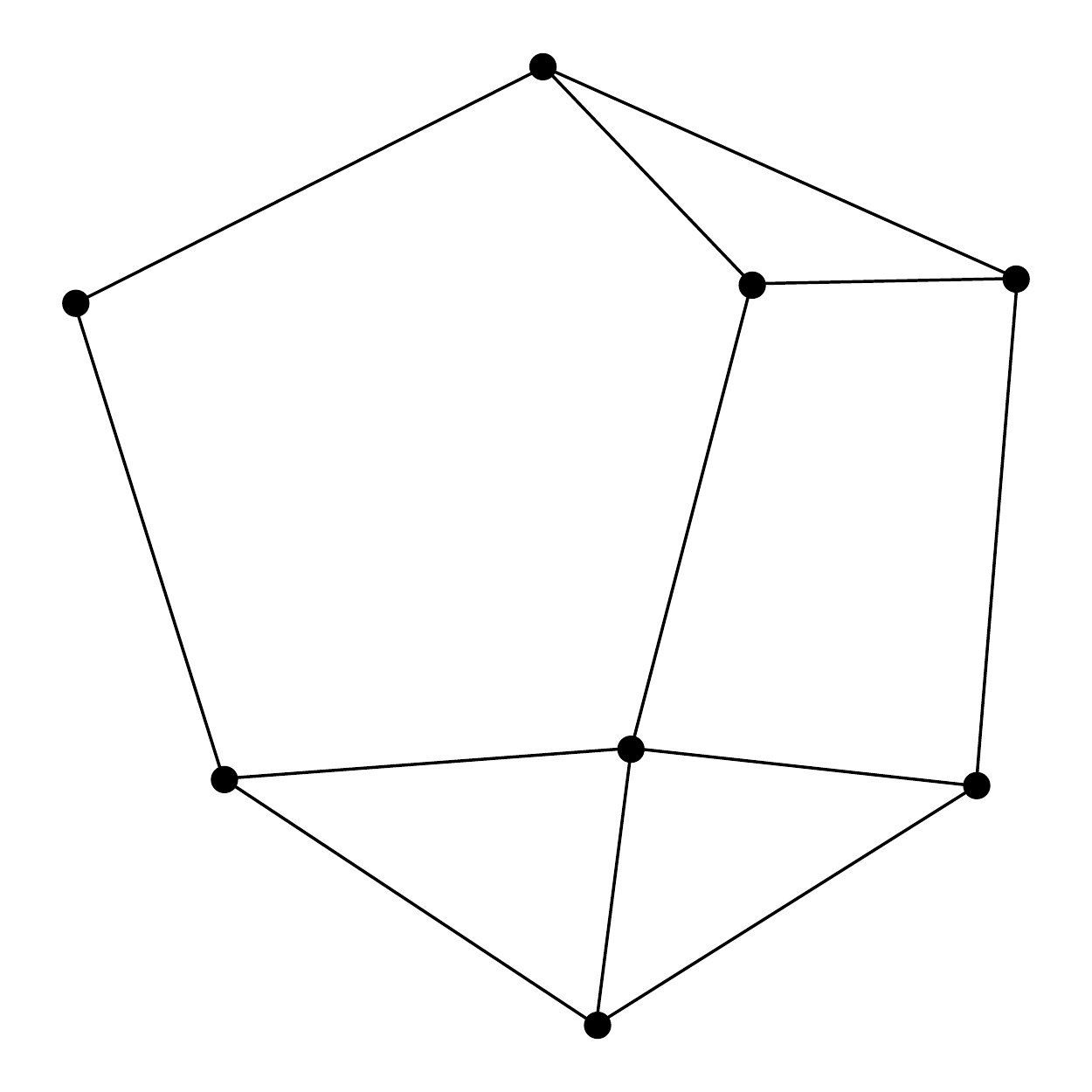}
   \includegraphics[width=0.5\textwidth]{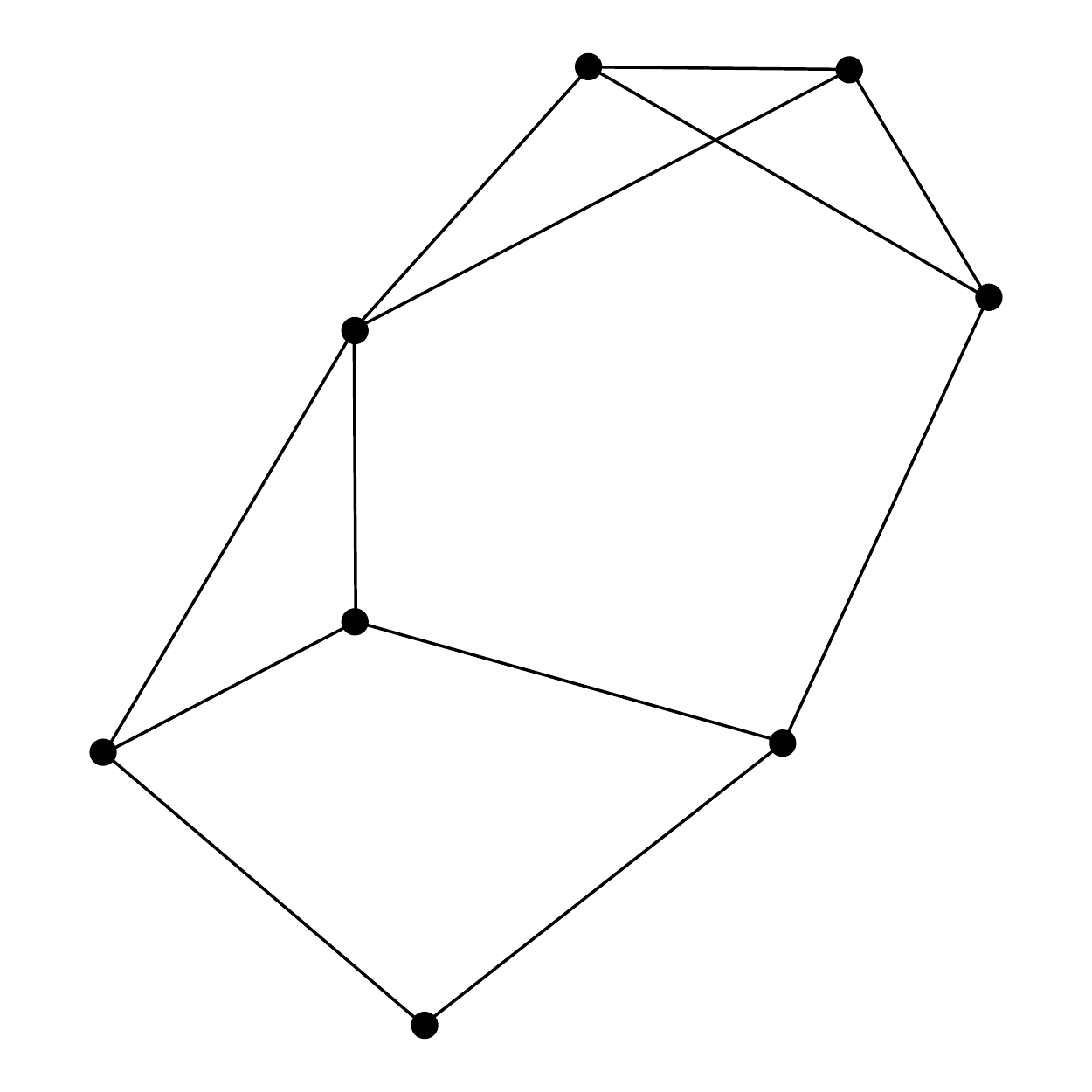}
    \caption{Two graphs $G_1$ and $G_2$ with $\mathcal{P}_2(G_1)=\mathcal{P}_2(G_2)$ and distinct chromatic polynomials.}\label{fig:chro}
\end{figure}

\begin{figure}
    \includegraphics[width=0.5\textwidth]{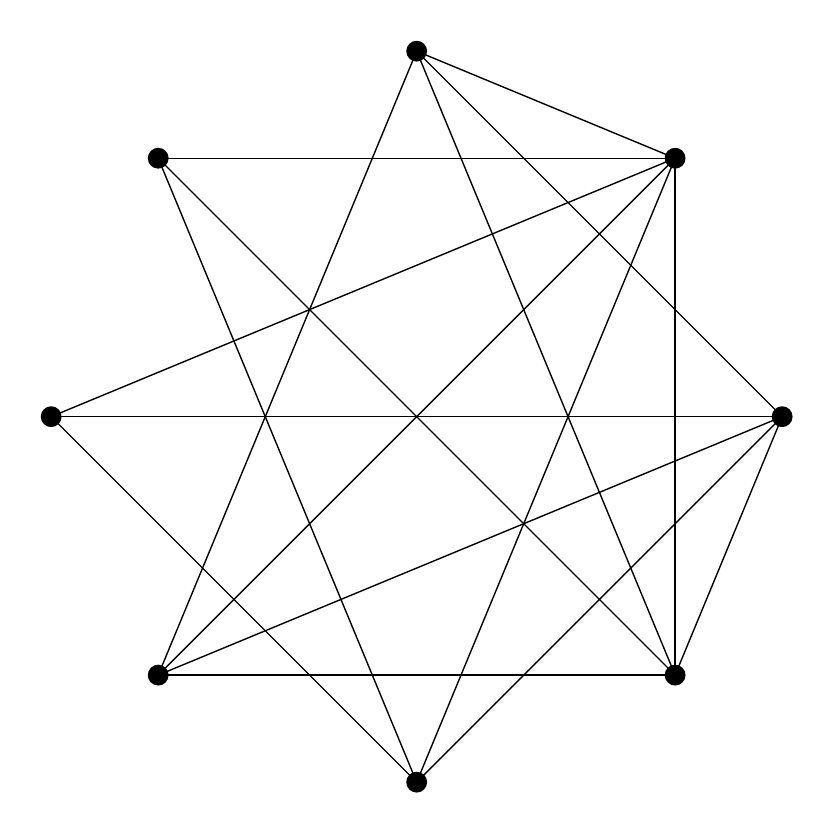}
   \includegraphics[width=0.5\textwidth]{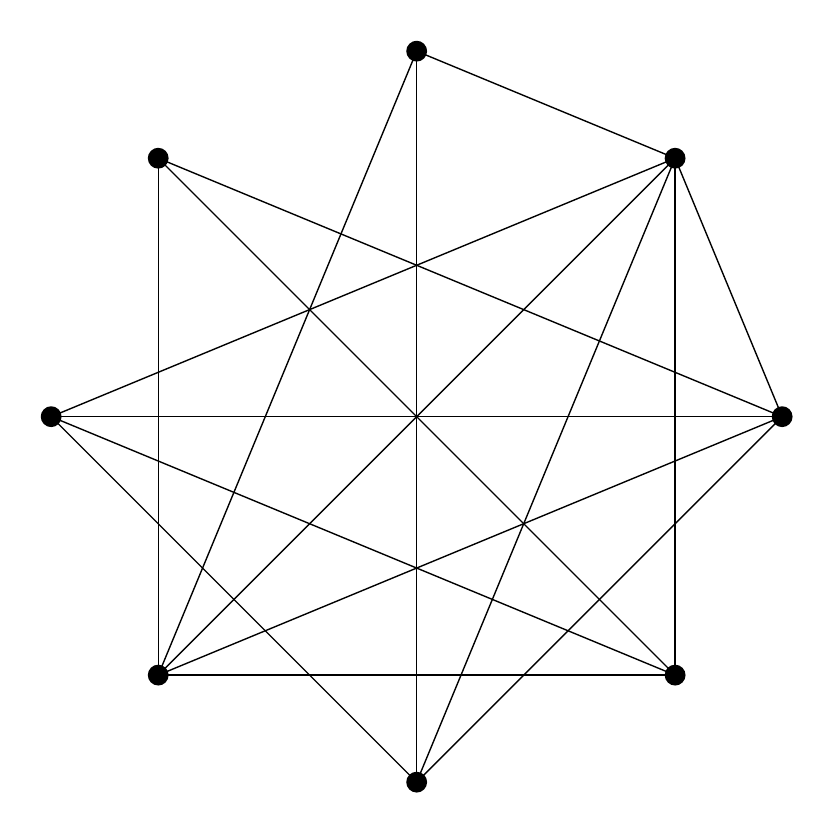}
    \caption{Two graphs $G_1$ and $G_2$ with the same  $q$-state Potts partitions function for  $q=2,3$ and distinct chromatic polynomials.}\label{fig:chro2}
\end{figure}

Another polynomial which is contained in the  Tutte-polynomial is the partition function of the $q$-state Potts model, here denoted by $P(G,q,y)$. This  polynomial is obtained from $\P(G)$ by setting $x_{i}=1$ for all $i$, $y_{ij}=1$ when $i<j$ and $x_{ii}=y$. 
\begin{lemma}
	The $q$-state Potts partition function is determined by $\P(G)$.
\end{lemma}
In \cite{GGN2} this polynomial was studied in detail and a number of graph families which can be determined from their Potts partition function for different values of $q$ were found. We refer to \cite{GGN2} for more details. 

In \cite{NW} Noble and Welsh defined the $U$-polynomial of a graph. This polynomial is a significant generalization of the Tutte-polynomial.  We will state one of the equivalent definitions given in \cite{NW}. First, given a set edges $A$ from a graph $G$, we let $G |A$ denote the subgraph obtained by deleting all edges not in $A$ from $G$, and we let $comp(G)$ denote the set of connected components of the graph $G$.
\begin{definition}
	$$U(G)=\sum_{A\subseteq E(G)} \left(y^{|A|-r(A)}\prod_{H\in comp(G|A)}x_{|H|}\right ),$$
	where $r(A)=|V(G)|-k(G|A)$, and $k(G|A)$ denotes the number of connected components in $G|A$.
\end{definition}
We note that if $G$ has $n$ vertices then $U(G)$ is a polynomial in $n$ variables, and each monomial  has total degree $n$ in the $x_i$'s. In \cite{NW} it was also proven that the $U$-polynomial is equivalent to the Tutte symmetric function introduced by Stanley in \cite{Stan}.  Tightly following this Sarmiento \cite{SI} proved that the $U$-polynomial, and hence the Tutte symmetric function as well, is equivalent to the so called polychromate, introduced by Brylawski \cite{Bry} in  1981.  Brylawski also proved that the polychromate of $G$ determines the polychromate of $\overline{G}$, and hence the same is true for the $U$-polynomial and the Tutte symmetric function.

One obvious difference between these three polynomials and the ones we have discussed earlier is that the number of variables is a function of the size of the graph. Hence, for a fixed $q$,  they are typically much larger objects than $\P(G)$.  Since the chromatic polynomial is a specialization of the $U$-polynomial and $\P(G)$ does not determine the chromatic polynomial we know that $P(G)$ cannot determine $U(G)$.  The reverse question is far from obvious, however it turns out that for large enough $q$ the two polynomials are not equvialent
\begin{observation}
	The two graphs in Figure \ref{fig:Upol} have the same $U$-polynomial but distinct homomorphism polynomials for $q=3$.
\end{observation}
By computing the $U$-polynomial for all graphs on $n\leq 8$ vertices we found  that the smallest graphs with the same $U$-polynomial have 8 vertices. There are 8 non-trivial equivalence classes, all of which are pairs. The graphs in each pair also have the same  homomorphism polynomial for $q=2$.
\begin{figure}
    \includegraphics[width=0.5\textwidth]{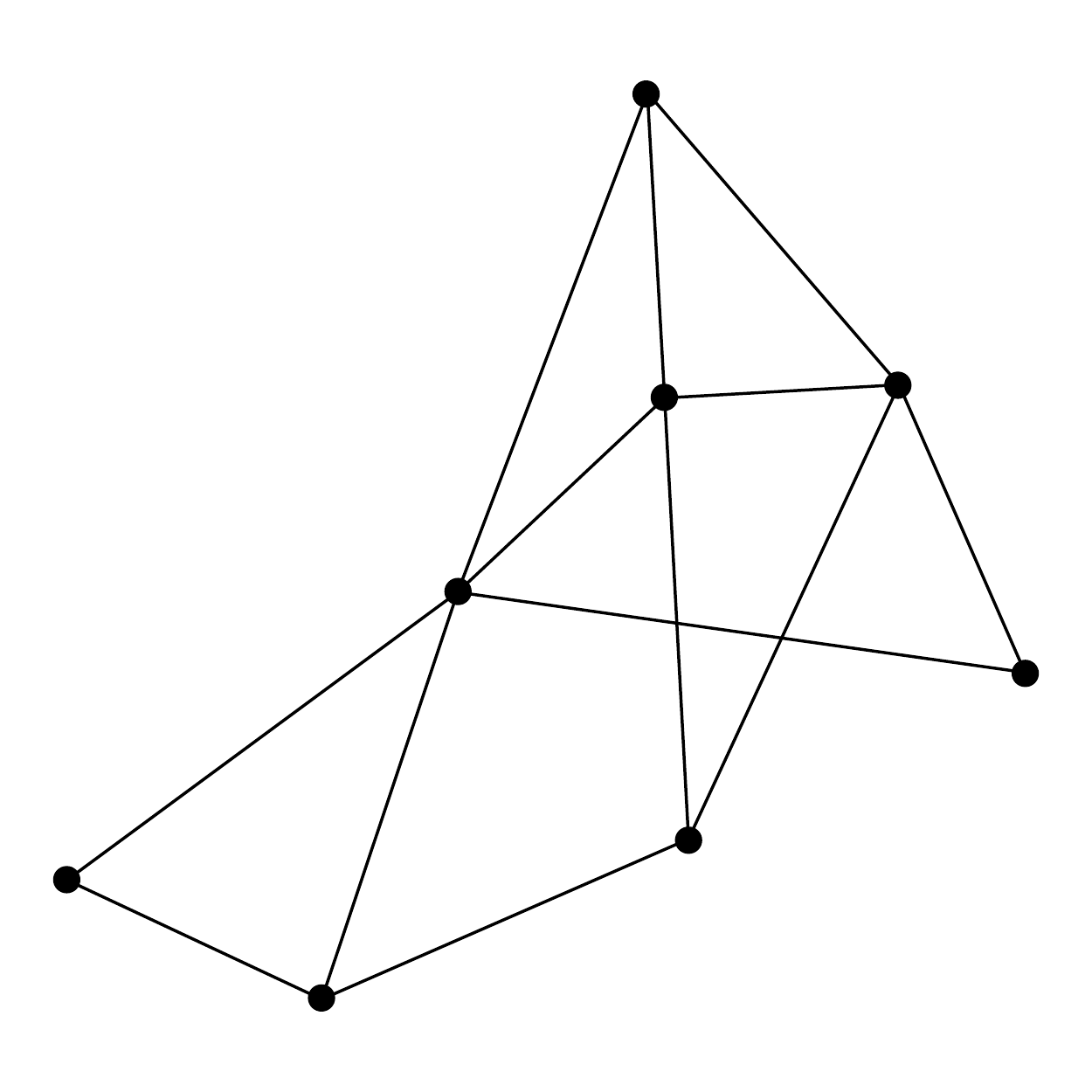}
   \includegraphics[width=0.5\textwidth]{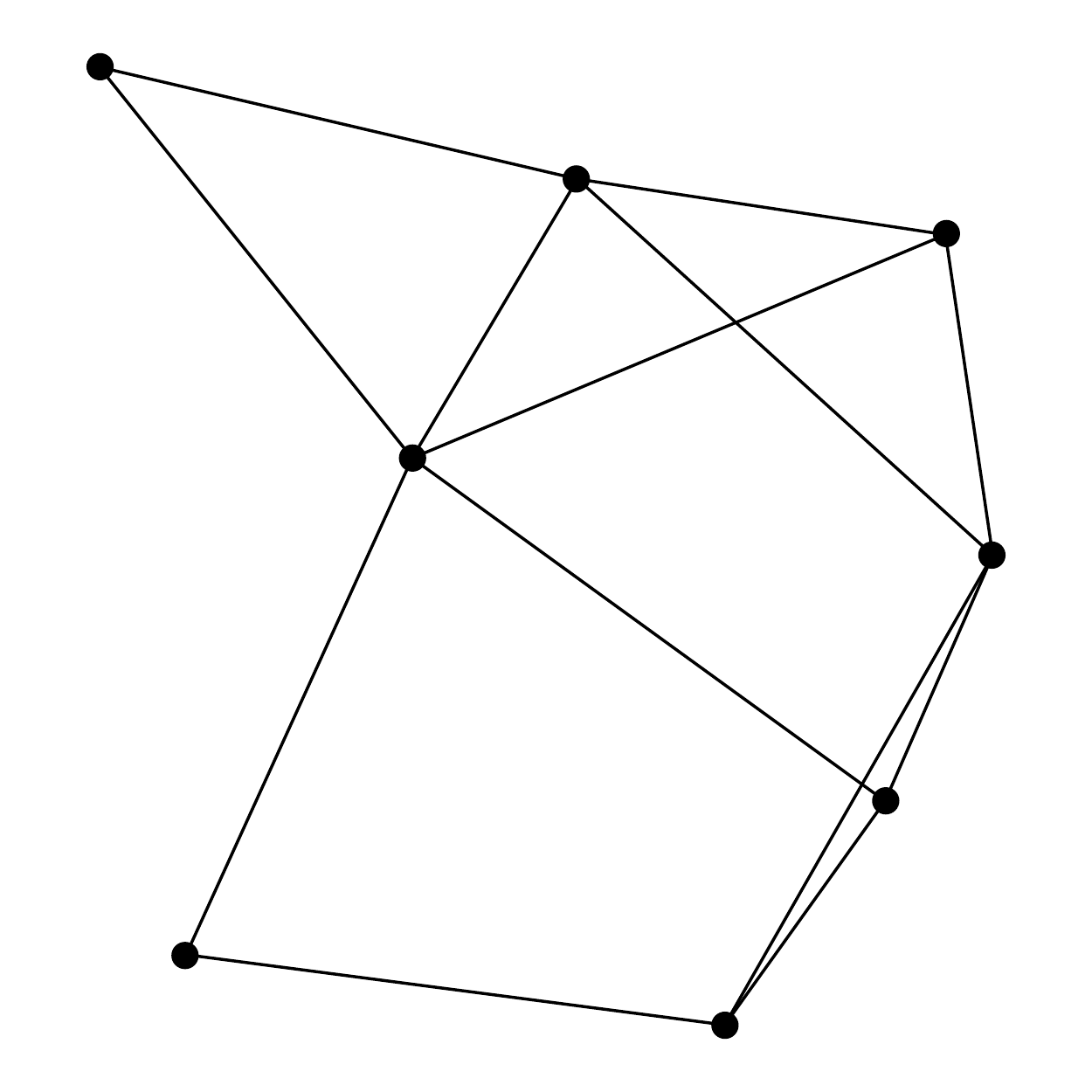} 
    \caption{Two 8-vertex graphs $G_1$ and $G_2$ with the same $U$-polynomial but  $\mathcal{P}_3(G_1)\neq \mathcal{P}_3(G_2)$.}\label{fig:Upol}
\end{figure}
Based on the last observation it is natural to ask the following question
\begin{problem}
	Does the $U$-polynomial determine $\mathcal{P}_2(G)$?
\end{problem}

In \cite{MN} the authors introduced what they called  strong versions of both the $U$-polynomial and the Tutte symmetric function and proved that both of these polynomials are equivalent to the strong polychromate, introduced by Bollob\'as and Riordan in \cite{BR:00}. They also showed that these polynomials determine Tutte's universal $V$-function \cite{Tu:47}, a very general polynomial with the property that if $G$ has several components then the $V$-function is the product of the $V$-functions of the components. 
If we let $SU(G)$ denote the strong $U$-polynomial of $G$ then, as defined in \cite{MN}, 
\begin{definition}
	$$SU(G)=\sum_{A\subseteq E(G)} \prod_{H\in comp(G|A)}x_{|H|,|E(H)-|H||+1}$$
\end{definition}
An  open problem from \cite{MN} is to find two graphs with the same $U$-polynomials but distinct strong $U$-polynomials. By extending our computation to the $U$-polynomial we found the following. 
\begin{observation}
	All pairs of graphs on $n\leq 10$ vertices which have the same $U$-polynomial also have the same strong $U$-polynomial.
\end{observation}

\section{Computational complexity and high-low symmetries for random graphs and tree-width}
It is well known that several of the polynomials from Theorem \ref{th3} are $\#P$-hard to compute and hence the homomorphism polynomial must be at least as hard to compute. The definition of $\P$ provides a straight forward algorithm for computing this polynomial in time $\mathcal{O}(poly(n)q^n)$ for an $n$-vertex graph $G$. However, as for many other computational problems, there are efficient algorithms when the input is restricted to graphs with bounded tree width.

Using the general dynamic programming methods described in \cite{reed} or \cite{SS5} it is straightforward to prove the following theorem for  graphs of bounded tree width 
\begin{theorem}
	For every fixed $q$ and $t$ there is an  algorithm with a running time which is polynomial in $n$ for computing $\P(G)$ for $n$-vertex graphs with tree width at most $t$. 
\end{theorem}
The multiplicative constant  in the run-time bound is of the form $\mathcal{O}(q^t)$. Here theorem \ref{th2} has an interesting consequence.
\begin{corollary}
	If $\overline{G}$ has tree width $t$ then $\P(G)$ can be computed in polynomial time.
\end{corollary}
We also note that the transfer matrix methods from \cite{LM08} can be used to compute $\P$ efficiently for so called poly-graphs, and also make it possible to use the automorphism group of $G$ to speed up the computation.

The behaviour of algorithms for computing NP-hard, and some polynomial, graphs properties on random graphs has been an active area of research in the last decade, with the main emphasis on the $k$-SAT problem.  In the series of papers by Scott and Sorkin mentioned earlier it was shown that for random graphs from the $G(n,p)$  model there is a threshold at $p=\frac{1}{n}$, such that for smaller $p$ there are algorithms with polynomial expected running time for several 2-CSP which in general are NP-hard, and for larger $p$ the same algorithms have an exponential expected running time.  

In \cite{LLO} it was proven that the tree-width of a random graph has a threshold 
\begin{theorem}
	Let $G$ be a graph from $G(n,p)$.
\begin{enumerate}
	\item  If $p=\frac{c}{n}$, with $c<1$ then $tw(G)\leq 2$ with probability $1-o(1)$. 
	\item If $p=\frac{c}{n}$, with $c>1$ then there exist a function $f(c)$ such that $tw(G) \geq f(c) n$ with probability $1-o(1)$.
\end{enumerate}
\end{theorem}
Hence the tree-width based algorithm for computing  $\mathcal{P}_q(G)$ has a polynomial median running time for $p<\frac{1}{n}$.  In order to prove that the expected running time is  polynomial as well, a concentration result for the tree-width of random graph ( which is stronger than what a simple application of e.g. Azuma's inequality gives) is needed. In particular the following problem becomes interesting:
\begin{problem}
	Let $p=\frac{c}{n}$ for some $c<1$ and $G$ be a graph from $G(n,p)$.  For which $q\geq 1$ is $\mathbb{E}(q^{tw(G)})=\mathcal{O}(n^k)$.?	
\end{problem}

\section{Further directions: directed graphs and quantum models}
There is a natural generalisation $\overrightarrow{\mathcal{P}}(D,q)$ of $\mathcal{P}(G,q)$  for directed multigraphs $D$, defined by replacing the weighted graph $W_q$ in the definition of $\mathcal{P}(G,q)$ by a complete directed graph on $q$ vertices, with both an edge from  $i$ to $j$ and one from $j$ to $i$ for each $i$ and $j$, where the weight $x_{i,j}$ is now distinct from $x_{j,i}$.

Our discussion of  $\mathcal{P}(G,q)$ has been based on the natural connection between graph  homomorphisms and graph properties related to various form of partitions of the vertex set of a graph. However, it is well known, see \cite{HN} for an in-depth discussion, that the theory of  homomorphisms of directed graphs  is in many ways better behaved than that for ordinary graphs, and hence $\overrightarrow{\mathcal{P}}(D,q)$ would seem to be an natural object to study. 

There has been several approaches to defining an analogue of the Tutte polynomial for directed graphs as well.  In  \cite{ChGr} one such generalisation was given in the form of the cover polynomial of a digraph, a polynomial which has also inspired work on colouring of ordinary graphs \cite{St}. As we have already done for the Tutte polynomial and the homomorphism polynomial it would be interesting to investigate which information the cover polynomial and the directed homomorphism polynomial 
$\overrightarrow{\mathcal{P}}(D,q)$  share.

As we have already mentioned $\mathcal{P}_2(G)$ contains the partition function of the Ising model, with an external field, as studied in  \cite{AM}. In physics  there is also a quantum version of the Ising model where there is also a so called transversal field involved, and this model has a partition function in three variables, two of which are the ones in the bivariate Ising polynomial and one which is associated with the transversal field. However for this model the partition function is not a polynomial. Nonetheless it still forms an invariant for graphs and has been studied in connection with adiabatic quantum algorithms in  \cite{quant}. There it was found that as an isomorphism invariant the quantum partition function is very strong, and in fact no example of two non-isomorphic graphs with the same partition function is known.  Finding such a pair would of course be interesting, and  it would also be interesting to consider a quantum version of $\mathcal{P}_q(G)$  as well. We refer the reader to \cite{quant} for further discussion of quantum models.


\newcommand{\etalchar}[1]{$^{#1}$}
\providecommand{\bysame}{\leavevmode\hbox to3em{\hrulefill}\thinspace}
\providecommand{\MR}{\relax\ifhmode\unskip\space\fi MR }
\providecommand{\MRhref}[2]{%
  \href{http://www.ams.org/mathscinet-getitem?mr=#1}{#2}
}
\providecommand{\href}[2]{#2}

\section*{Appendix: The computer search}\label{comp}

Computing the homomorphism polynomial $\P(G)$ for a general graph $G$ is $\#P$-hard, and in practice we have no efficient general algorithm for this problem, 
even for small fixed $q$. So, in order to reduce the computational burden we have performed our search in steps.

Let us assume that we have some graph invariant $F(G)$ and we wish to find all non-trivial equivalence classes of $n$-vertex graphs for $F$.  Let us also assume that 
we have some sequence of easier to compute graph invariants $f_1,f_2,\ldots, f_k$, all of which have that property that if  $F(G_1)=F(G_2)$ then $f_i(G_1)=f_i(G_2)$.  We now proceed as follows:

\begin{enumerate}
	\item Set $i=1$ and let our collection of equivalence classes $S$ consist of a single equivalence class consisting of all graphs on $n$ vertices.
	
	\item Partition each equivalence class in $S$ into equivalence classes for $f_i$ and let $S$ be the list of all such classes. 
	
	\item Delete all singleton classes from $S$,  increase $i$ by 1 and repeat from step 2 unless $i>k$ 
	
	\item Compute $\P(G)$ for all remaining graphs and partition the remaining classes in $S$ into equivalence classes for $\P(G)$ 

\end{enumerate}
In order to make this procedure as fast as possible we want to  choose a list of graph invariant such that this with low index are very fast to compute, since they 
will be computed fore the largest number of graphs, and hard to compute invariants will only be included for high indices if they have a significant ability to break 
the remaining equivalence classes into smaller ones. Ideally we want to find achieve singleton classes as early as possible.

In our search for graphs with equal homomorphism polynomials we first used the following invariants, in this order:
\begin{enumerate}
	\item The number of edges.
	
	\item The degree sequence.
	
	\item The pair: the number of triangles and the edge-connectivity of $G$.
	
	\item The matching polynomial.
	
	\item The pair: the independence polynomial of $G$ and the independence polynomial of the complement of $G$.
	
	\item The bivariate Ising polynomial.
	
	\item The homomorphism polynomial for $q=2$.
	
\end{enumerate}

Before going on to the case $q=3$ for $\P(G)$ we also partitioned the $q=2$ equivalence classes using the number of proper 3-colourings as an invariant, and before proceeding to $q=4$ we did the same for the $q=3$ equivalence classes using the number of proper 4-colourings.

\end{document}